\documentclass[11pt]{amsart}

\usepackage{amsmath,amssymb,graphicx,bbm}
\usepackage{amsthm,verbatim}
\usepackage{mathrsfs,mathtools}
\usepackage[linesnumbered,ruled,lined,resetcount]{algorithm2e}

\usepackage[footnotesize,bf]{caption}
\usepackage[left=1.25in,right=1.25in,top=1in]{geometry}

\usepackage{array,multirow,booktabs} 

\usepackage{mathdefs}
\usepackage[dvipsnames]{xcolor}

\newcommand{\DP}{\texttt{DP}}
\newcommand{\HD}{\texttt{HD}}
\newcommand{\MC}{\texttt{MC}}
\newcommand{\aPC}{\texttt{aPC}}
\newcommand{\SP}{\texttt{SP}}
\newcommand{\LZ}{\texttt{LZ}}
\newcommand{\PC}{\texttt{PC}}
\newcommand{\PCL}{\texttt{PCL}}
\newcommand{\bs}[1]{\boldsymbol{#1}}

\usepackage{url}

\author{Zexin Liu}\thanks{This work was supported by the National Institute of Biomedical Imaging and Bioengineering of the National Institutes of Health under grant number U24EB029012, and under National Science Foundation award DMS-1720416.}
\author{Akil Narayan}
\address{Department of Mathematics, and Scientific Computing and Imaging (SCI) Institute, the University of Utah}
\email{zexin@math.utah.edu, akil@sci.utah.edu}

\title{On the computation of recurrence coefficients for univariate orthogonal polynomials}

\begin{document}
\begin{abstract}
Associated to a finite measure on the real line with finite moments are recurrence coefficients in a three-term formula for orthogonal polynomials with respect to this measure. These recurrence coefficients are frequently inputs to modern computational tools that facilitate evaluation and manipulation of polynomials with respect to the measure, and such tasks are foundational in numerical approximation and quadrature. Although the recurrence coefficients for classical measures are known explicitly, those for nonclassical measures must typically be numerically computed. We survey and review existing approaches for computing these recurrence coefficients for univariate orthogonal polynomial families and propose a novel ``predictor-corrector" algorithm for a general class of continuous measures. We combine the predictor-corrector scheme with a stabilized Lanczos procedure for a new hybrid algorithm that computes recurrence coefficients for a fairly wide class of measures that can have both continuous and discrete parts. We evaluate the new algorithms against existing methods in terms of accuracy and efficiency.

\end{abstract}

\maketitle

\section{Introduction}\label{sec:introduction}

Univariate orthogonal polynomials are a mainstay tool in numerical analysis and scientific computing. These polynomials serve as theoretical foundations for numerical algorithms involving approximation and quadrature \cite{szego_orthogonal_1975,freud_orthogonal_1971,nevai_orthogonal_1980,gautschi_orthogonal_2004,gautschi_orthogonal_2006}. Given a positive measure $\mu$ on the real line $\R$, if $\mu$ has finite polynomial moments of all orders along with an infinite number of points of increase, then a family of orthonormal polynomials $\{p_n\}_{n=0}^\infty$ exists, satisfying $\deg p_n = n$, and
\begin{align*}
  \int_\R p_n(x) p_m(x) \dx{\mu}(x) = \delta_{m,n},
\end{align*}
where $\delta_{m,n}$ is the Kronecker delta. If we further assume that each $p_n$ has a positive leading coefficient, then these polynomials are unique. Such families are known to obey a three-term recurrence formula,
\begin{align}\label{eq:ttr}
  x p_n(x) &= b_n p_{n-1}(x) + a_{n+1} p_n(x) + b_{n+1} p_{n+1}(x), & n \geq 0,
\end{align}
with the starting conditions $p_{-1} \equiv 0$ and $p_0(x) = 1/b_0$. The coefficients $(a_n)_{n=1}^\infty \subset \R$ and $(b_n)_{n=0}^\infty \subset (0, \infty)$ depend only on the (polynomial) moments of $\mu$. In practical settings, knowledge of these coefficients is the only requirement for implementing stable, accurate algorithms that achieve evaluation and manipulation of polynomials that are core components of approximation and quadrature algorithms. For example, the $n$ eigenvalues of the $n \times n$ Jacobi matrix $\bs{J}_n$ are precisely the abscissae of a $\mu$-Gaussian quadrature rule, with $\bs{J}_n$ the symmetric tridiagonal matrix given by
\begin{align}\label{eq:jacobi-matrix}
  \bs{J}_n(\mu) = \left(\begin{array}{ccccc} 
    a_1 & b_1 & & & \\
    b_1 & a_2 & b_2 & & \\
    & \ddots & \ddots & \ddots & \\
    & & b_{n-2} & a_{n-1} & b_{n-1} \\
    & & & b_{n-1} & a_n \\
  \end{array}\right).
\end{align}
Therefore, the recurrence coefficients $a_n$ and $b_n$ must be computed stably and accurately. 

Some classical probability measures $\mu$ give rise to classical families of orthogonal polynomials $p_n$: A Gaussian measure results in Hermite polynomials; the uniform measure on a compact interval results in Legendre polynomials; a Beta measure corresponds with Jacobi polynomials; and a one-sided exponential measure gives rise to Laguerre polynomials. These classical polynomial families are among a few for which explicit formulas are available for the recurrence coefficients $a_n$ and $b_n$, see, e.g., \cite[Tables 1.1, 1.2]{gautschi_orthogonal_2004}.

However, for even modestly complicated measures $\mu$ outside this classical collection, the task of determining these coefficients can be quite difficult. For example, an application in which this situation arises is in polynomial Chaos methods, which are techniques in scientific computing problems for modeling the effect of uncertainty in a model \cite{wiener_homogeneous_1938,xiu_wiener--askey_2002}. An output's dependence on a finite number of random variable inputs is modeled with polynomial dependence on those inputs. With one random input, the polynomial approximation is typically constructed using a basis of polynomials orthogonal to the distribution of the random input, which requires building orthogonal polynomials with respect to a given, often nonclassical, probability measure.

A simple example that illustrates how computation of orthogonal polynomials is difficult for even fairly simple measures is furnished by the class of \textit{Freud weights},
\begin{align}\label{eq:freud}
  \dx{\mu}(x) &= \exp\left( - |x|^\alpha \right) \dx{x}, & \alpha &> 0,
\end{align}
with support equal to all of $\R$.  (In what follows, we will refer to $\mu$ as a measure and $\dx{\mu}$ as a weight.) When $\alpha = 2$, corresponding to the Gaussian measure (and Hermite polynomial family), the three-term recurrence coefficients are known exactly. However, when $\alpha = 1$, no closed-form analytical formula for the coefficients $a_n$ and $b_n$ exists, even though the moments of $\mu$ are known explicitly in terms of well-studied special functions. (For example, note that under a change of variable, the moments of the measure above correspond to evaluations of the Euler Gamma function.) 

In such general cases when no known closed-form expression for the three-term recurrence coefficients exists, numerical methods are employed to approximate them. The main goal of this article is to survey and extend existing methods for computing these recurrence coefficients associated to measures for which explicit formulas are not available.

\subsection{Existing approaches}

When $\mu$ is not a measure for which the coefficients have explicitly known formulas, one typically resorts to numerical methods to approximately compute these coefficients. A summary of the methods we consider in this article is presented in Table \ref{tab:notation}, which indicates later sections in this article where we give a formal description of each algorithm. A brief description of these procedures is given in Section \ref{sec:methods}, but an excellent and more detailed historical survey is provided in \cite[Section 2.6]{gautschi_orthogonal_2004}. Below we present a nontechnical summary of the approaches that we survey.

A classical approach to computing recurrence coefficients from moments is via determinants of Hankel matrices \cite[Section 2.1.1]{gautschi_orthogonal_2004}. A second classical approach, the Chebyshev algorithm, transforms monomial moments by expressing the recurrence coefficients in terms of moments involving monomials and $p_n$ \cite{chebyshev1859interpolation}. A more effective approach, the modified Chebyshev algorithm, uses moments involving $p_n$ and another arbitrary set of polynomials \cite{sack_algorithm_1971,wheeler_modified_1974,gautschi_survey_1981}. Yet another procedure, the Stieltjes algorithm \cite{stieltjes1884quelques}, computes recurrence coefficients directly assuming moments involving $p_n$ can be computed. Finally, given a measure with discrete support, the Lanczos algorithm can be used to compute the Jacobi matrix for $\mu$, yielding the recurrence coefficients; although this is typically unstable, a stable variant is given in \cite{rutishauser1963jacobi}. 

For very special forms of weight functions, other procedures can be derived. A primary example of this are iterative recurrence-type algorithms resulting from discrete Painlev\'{e} equations when $\dx{\mu}(x) \propto \exp(-x^{\alpha})$ for $\alpha/2 \in \N$. These Painlev\'{e} equations, which determine the recurrence coefficients for $p_n$, are remarkably simple and direct to implement, but are quite unstable \cite{assche_discrete_2005}. A final approach we consider amounts to using a linear orthogonalization procedure, such as (modified) Gram-Schmidt, to compute the expansion coefficients of $p_n$ in terms of the monomials. However, this procedure is known to produce quite ill-conditioned matrices, especially for large $n$, making the computation of $p_n$, and hence the recurrence coefficients, suffer roundoff errors. Therefore, although this approach has often been used \cite{witteveen2006modeling,witteveen2007modeling}, it is less useful in the context of this article. Nevertheless, we consider one recent related approach, an ``arbitrary polynomial chaos" approach suggested in \cite{oladyshkin_data-driven_2012}, which amounts to solving a linear system involving a modified Hankel matrix.

\begin{table}
  \begin{center}
  \resizebox{1.0\textwidth}{!}{
    \renewcommand{\tabcolsep}{0.4cm}
    \renewcommand{\arraystretch}{1.3}
    {\scriptsize
      \begin{tabular}{cccccc} 
      \toprule
        Method & Abbreviation & Section & Citation \\\midrule
         Discrete Painlev\'{e} I equations method & \DP & \ref{subsec:DP} & \cite{assche_discrete_2005} \\
         Hankel Determinants & \HD & \ref{subsec:HD} & \cite[Section 2.1.1]{gautschi_orthogonal_2004} \\
         Arbitrary polynomial chaos expansion method & \aPC & \ref{subsec:aPC} & \cite[Section 3.1]{oladyshkin_data-driven_2012} \\
         Modified Chebyshev algorithm & \MC & \ref{subsec:MC} & \cite[Section 2.1.7]{gautschi_orthogonal_2004} \\
         Stieltjes procedure & \SP & \ref{subsec:SP} & \cite[Section 2.2.3.1]{gautschi_orthogonal_2004} \\
         Stabilized Lanczos algorithm & \LZ & \ref{subsec:LZ} & \cite[Section 2.2.3.2]{gautschi_orthogonal_2004} \\
         Predictor-corrector method & \PC & \ref{subsec:PC} & --- \\
         Predictor-corrector-Lanczos method & \PCL & \ref{subsec:PCL} & --- \\
      \bottomrule
      \end{tabular}
    }
 }
  \end{center}
  \caption{Abbreviation, subsection, and algorithm for each method. Also included is a modern citation that explains each algorithm}\label{tab:notation}
\end{table}

\subsection{Contributions of this article}
Several algorithms exist to compute the recurrence coefficients, but a few clear and direct recommendations are available for researchers without substantial experience and/or knowledge of the field. The main contribution of this paper is to summarize, evaluate, and extend existing methods for computing recurrence coefficients for univariate orthogonal polynomial families. We first provide a survey and comparison of many existing algorithms (see Section \ref{sec:methods}). In Section \ref{subsec:PC} we propose a novel ``predictor-corrector" algorithm and evaluate its utility. Finally, by modifying the ``multiple component" approach in \cite{gautschi_generating_1982,gautschi_algorithm_1994}, we consider a new hybrid algorithm in Section \ref{subsec:PCL} that combines our predictor-corrector scheme with a stabilized Lanczos procedure. Our algorithm can be used to compute recurrence coefficients for the fairly general class of measures whose differentials are given by
\begin{align}\label{eq:mu}
  \dx{\mu}(x) = \sum_{j=1}^C w_j(x) \mathbbm{1}_{I_j}(x) \dx{x} + \sum_{j=1}^M \nu_j \delta_{\tau_j} \dx{x},
\end{align}
where $C$ and $M$ are finite (either possibly 0), $\delta_{\tau_j}$ is a Dirac mass located at $\tau_j \in \R$, $\{\nu_j\}_{j=1}^M$ are positive scalars, each $I_j$ is a (possibly unbounded) nontrivial interval, and $w_j$ is a continuous (ideally smooth) non-negative function on $I_j$. Specification of the $w_j$, $I_j$, $\tau_j$, and $\nu_j$ is sufficient to utilize most of the algorithms we consider, but having extra information that characterizes $w_j$, particularly prescribed behavior at finite endpoints of $I_j$, will increase the accuracy of the procedures. In other words, with $I_j = [\ell_j, r_j]$ and either of the endpoints $\ell_j, r_j$ is finite, we assume knowledge of exponents $\beta_j, \alpha_j > -1$ such that $w_j$ has polynomial singular strength $\beta_j$, $\alpha_j$ at endpoints $\ell_j, r_j$, i.e., 
\begin{align}\label{eq:singularity-behavior}
  0 < &\lim_{x \downarrow \ell_j} w_j(x) (x - \ell_j)^{-\beta_j} < \infty, & 
  0 < &\lim_{x \uparrow r_j} w_j(x) (r_j - x)^{-\alpha_j} < \infty.
\end{align}
Note that our assumption that $\alpha_j, \beta_j > -1$ is natural since if the inequality above is true with, say, $\alpha_j \leq -1$, then $\mu$ is not a finite measure and therefore is not a probability measure.

Note that the form of $\mu$ we assume in \eqref{eq:mu} is quite general, and includes all classical measures, those with piecewise components, measures with discrete components, measures with unbounded support, and measures whose densities have integrable singularities. 

This paper is structured as follows: In section \ref{sec:methods} we briefly survey the existing approaches summarized in Table \ref{tab:notation}. Section \ref{sec:hybrid} contains the discussion that leads to our proposed hybrid ``\PCL" algorithm: Section \ref{subsec:PC} discusses the predictor-corrector scheme; section \ref{subsec:moments} briefly describes how we compute moments, which leverages the specific form of the measure $\mu$ assumed in \eqref{eq:mu} and \eqref{eq:singularity-behavior}; section \ref{subsec:PCL} combines these with a stabilized Lanczos procedure. Finally, we present a wide range of numerical examples in Section \ref{sec:numerical}, which compares many of the techniques in Table \ref{tab:notation}, and demonstrates the accuracy and efficiency of the ``\PCL" algorithm.

\section{Existing approaches}\label{sec:methods}

We review here some existing methods for computing recurrence coefficients. In order to compute the required coefficients, having some knowledge about the measure $\mu$ is neccessary. The following are two of the more common assumptions that one makes, with the latter assumption being stronger:
\begin{itemize}
  \item The (monomial) moments of all orders of $\mu$ are known, i.e., the moment sequence
    \begin{align}\label{eq:moments}
      m_n &\coloneqq \int x^n \dx{\mu}(x), & n &\geq 0,
    \end{align}
    is known and available. In practice, the integrals can be obtained by the composite quadrature approach introduced in Section \ref{subsec:moments}, but sometimes they can also be computed directly in terms of special functions, such as Gamma function given the Freud weights.
  \item General polynomial moments, i.e., 
    \begin{align}\label{eq:polymoments}
      \int q(x) \dx{\mu}(x),
    \end{align}
    are computable for a general, finite-degree polynomial $q$ that is often identified only partway through an algorithm.
\end{itemize}

No particular prescription exists for how the moments above are computed, but typically this is accomplished through a quadrature rule. In some ``data-driven" scenarios, this quadrature rule often comes as a Monte Carlo rule from an empirical ensemble.

We discuss six procedures below; in practice, only the last two are computationally stable, but they are all useful for comparison purposes. The first procedure works only for very special Freud weights, i.e., those with exponential behavior. 


\subsection{\DP: Freud weights and discrete Painlev\'{e} equations}\label{subsec:DP}\hfill

Freud weights, named after G\'{e}za Freud who studied them in the 1970s \cite{freud1976coefficients}, have the following form:
\begin{align}\label{eq:freud_weight}
\dx{\mu}(x) &= |x|^\rho \exp(-|x|^\alpha) \dx{x}, & \rho &> -1, \alpha > 0.
\end{align}
Observe that Freud weights are symmetric, which implies that $a_n = 0$ for $n \geq 0$, and therefore only the $b_n$ coefficients need be computed. Freud gave a recurrence relation for the recurrence coefficients $b_n$ when $\alpha = 2, 4, 6$. 
The connection between Freud weights and discrete Painlev\'{e} equations was first pointed out by Magnus \cite{magnus1996preud}. 
In the case of $\alpha = 4$, one can derive the following recurrence relation for $n \geq1$ by letting $x_n \coloneqq 2 b_n^2$:
\begin{align}\label{eq:recursion_freud4}
  x_{n+1} &= \frac{1}{x_n} \left( n + \frac{\rho}{2} \left(1 + (-1)^n\right)\right) - x_n - x_{n-1}, & x_0 &= 0, & \quad x_1 &= \frac{2 \Gamma{(\frac{3+\rho}{4}})}{\Gamma{(\frac{1+\rho}{4}})}.
\end{align}
See, e.g., \cite[Section 2.2]{assche_discrete_2005}. This recurrence relation is a discrete Painlev\'{e} I equation \cite{magnus1996preud} that is useful for theoretical analysis. For example, it can be used to prove Freud's conjecture, which is a statement about asymptotic behavior of the $b_n$ coefficients. For $\alpha = 4$ in this section, Freud's conjecture states
\begin{equation}\label{eq:conjecture_freud4}
\lim_{n\rightarrow\infty} \frac{b_n}{n^{1/4}} = \frac{1}{\sqrt[4]{12}}.
\end{equation}
A more general resolution of Freud's conjecture using alternative methods is provided in \cite{lubinsky_proof_1988}.

Similarly, when $\alpha = 6$, by letting $y_n \coloneqq b_n^2$, a fourth-order nonlinear recurrence relation for $n\geq2$ \cite[Section 2.3]{assche_discrete_2005} is given by
\begin{equation}\label{eq:recursion_freud6}
\begin{split}
6 y_n \left(y_{n-2} y_{n-1} + y_{n-1}^2 + 2 y_{n-1} y_n + y_{n-1} y_{n+1} + y_n^2 + 2 y_n y_{n+1} + y_{n+1}^2 + y_{n+1} y_{n+2}\right) \\
= n + \frac{\rho}{2} \left(1 + (-1)^n\right),
\end{split}
\end{equation}
with initial condition
\begin{align*}
y_0 &= 0, & y_1 &= \frac{\Gamma{(\frac{3+\rho}{6})}}{\Gamma{(\frac{1+\rho}{6})}}, & \\
 y_2 &= \frac{\Gamma{(\frac{5+\rho}{6})}}{\Gamma{(\frac{3+\rho}{6})}} - y_1, &
 y_3 &= \frac{\Gamma{(\frac{7+\rho}{6})}}{y_2 y_1 \Gamma{(\frac{1+\rho}{6})}} - \frac{2(y_1+y_2)\Gamma{(\frac{5+\rho}{6})}}{y_2 y_1 \Gamma{(\frac{1+\rho}{6})}} + \frac{(y_1+y_2)^2\Gamma{(\frac{3+\rho}{6})}}{y_2 y_1 \Gamma{(\frac{1+\rho}{6})}}.
\end{align*}
In this case, Freud's conjecture states
\begin{equation}\label{eq:conjecture_freud6}
\lim_{n\rightarrow\infty} \frac{b_n}{n^{1/6}} = \frac{1}{\sqrt[6]{60}}.
\end{equation}

Note the computation of recursion coefficients via \eqref{eq:recursion_freud4} and \eqref{eq:recursion_freud6} is quite straightforward, but is also very unstable. Nevertheless, there is a unique positive solution \cite{lew1983nonnegative}; hence, a small (e.g., machine roundoff) error in $x_1$ or $y_1$ quickly results in the loss of positivity of $x_n$ or $y_n$. Numerical solutions follow the exact asymptotic behavior well until large deviations from the true solution eventually appear, cf. Figure \ref{fig:freud}.

\subsection{\HD: Hankel Determinants}\label{subsec:HD}
Orthogonal polynomials as well as their recursion coefficients are expressible in determinantal form in terms of the moments of the underlying measure. Indeed, much of the classical theory of orthogonal polynomials is moment-oriented. One classical technique to express recurrence coefficients in terms of moments is via matrix determinants.

We introduce the Hankel determinant $\Delta_n$ of order $n$ in terms of the finite moments \eqref{eq:moments}, defined as
\begin{align}\label{eq:hankel_det}
\Delta_{-1} &= 1, &
\Delta_{0} &= 1, &
\Delta_n &= \det \bs{H}_n, &
\bs{H}_n &\coloneqq
\left(\begin{array}{cccc}
m_0 & m_1 & \cdots & m_{n-1} \\
m_1 & m_2 & \cdots & m_{n} \\
\vdots & \vdots & \ddots & \vdots \\
m_{n-1} & m_{n} & \cdots & m_{2n-2}
\end{array}\right), &
n &\in \N.
\end{align}

These determinants of Gram matrices are associated to the $\mu$-inner product, using a basis of monomials. In addition, we define determinants $\Delta_n^\prime$ of modified Hankel matrices, where the modification is to replace the last column of $\bs{H}_n$ by the last column of $\bs{H}_{n+1}$ with the trailing entry removed,
\begin{align*}
\Delta_0^{\prime} &= 0, &
\Delta_1^{\prime} &= m_1, &
\Delta_n^{\prime} &= 
\begin{vmatrix}
m_0 & m_1 & \cdots & m_{n-2} & m_{n} \\
m_1 & m_2 & \cdots & m_{n-1} & m_{n+1} \\
\vdots & \vdots & \vdots & \vdots & \vdots \\
m_{n-1} & m_{n} & \cdots & m_{2n-3} & m_{2n-1}
\end{vmatrix}, & n &= 2, 3, \dots \ .
\end{align*}
Along with $b_0 = \sqrt{m_0}$, the orthogonal polynomial recurrence coefficients can be computed explicitly from these determinants, cf.\cite[Theorem 2.2]{gautschi_orthogonal_2004},
\begin{align}\label{eq:ab_HD}
  a_n &= \frac{\Delta_{n}^{\prime}}{\Delta_{n}} - \frac{\Delta_{n-1}^{\prime}}{\Delta_{n-1}}, &
  b_n &= \sqrt{\frac{\Delta_{n+1} \Delta_{n-1}}{\Delta_n^2}}, &
  n\in \N.
\end{align}

The formulas \eqref{eq:ab_HD} are not practically useful as an algorithm to compute reucrrence coefficients since the Hankel matrices above are typically ill-conditioned. In particular, the map that computes recurrence coefficients from moments can be severely ill-conditioned \cite[Section 2.1.6]{gautschi_orthogonal_2004}.

\subsection{\aPC: ``Arbitrary" polynomial chaos expansions}\label{subsec:aPC}

The arbitrary polynomial chaos (\aPC{}), like all polynomial chaos expansion techniques, approximates the dependence of simulation model output on model parameters by expansion in an orthogonal polynomial basis. As shown in \cite{oladyshkin_data-driven_2012}, \aPC{} at finite expansion order demands the existence of only a finite number of moments and does not require the complete knowledge of a probability density function. Once we construct the polynomials such that they form an orthonormal basis for arbitrary distributions from the moment-based analysis, the recurrence coefficients can be derived using the \aPC{} expansion coefficients.

Our goal is, firstly, to construct the polynomials in \eqref{eq:expand_aPC} such that they form an orthonormal basis for arbitrary distributions. Instead of the normality condition, we will first introduce an intermediate auxiliary condition by demanding that the leading coefficients of all polynomials be equal to $1$.

We define the monic orthogonal polynomial $\pi_n(x)$ as
\begin{equation}\label{eq:expand_aPC}
  \pi_n(x) = \sum_{i=0}^n c_i^{(n)} x^i,
\end{equation}
where $c_i^{(n)}$ are expansion coefficients, and specifically, $c_n^{(n)} = 1, \forall n$. The general conditions of orthogonality for $\pi_n(x)$ with respect to all lower order polynomials can be written in the following form \cite[Section 3.1]{oladyshkin_data-driven_2012}:
\begin{align}\label{eq:lower_aPC}
  \int_{\Omega} x^k \left(\sum_{i=0}^n c_i^{(n)} x^i \right) \dx{\mu}(x) &= 0, & k &= 0, 1, \dots, n-1.
\end{align}
For each $n$, the system of equations given by \eqref{eq:lower_aPC} defines the unknown polynomial expansion coefficients in \eqref{eq:expand_aPC}. Using finite moments in \eqref{eq:moments}, the system can be reduced to
\begin{align*}
  \sum_{i=0}^n c_i^{(n)} m_{i+k} &= 0.
\end{align*}
Alternatively, the system of linear equations can be written in the more convenient matrix form,
\begin{equation}\label{eq:cnk-system}
\left(\begin{array}{cccc}
m_0 & m_1 & \cdots & m_{n} \\
m_1 & m_2 & \cdots & m_{n+1} \\
\vdots & \vdots & \ddots & \vdots \\
m_{n-1} & m_{n} & \cdots & m_{2n-1} \\
0 & 0 & \cdots & 1
\end{array}\right)
\left(\begin{array}{c}
c_0^{(n)} \\
c_1^{(n)} \\
\vdots \\
c_{n-1}^{(n)} \\
c_n^{(n)}
\end{array}\right) = 
\left(\begin{array}{c}
0 \\
0 \\
\vdots \\
0 \\
1
\end{array}\right).
\end{equation}
By defining the coefficient vector ${\bs{c}^{(n)}} = \left(c_0^{(n)}, c_1^{(n)}, \dots, c_n^{(n)}\right)^T$, the normalized coefficients $\bar{c}_i^{(n)}$ can be expressed in terms of $\bs{c}^{(n)}$ and Hankel matrices $\bs{H}_{n+1}$,
\begin{align}\label{eq:normalc}
  \bar{c}_i^{(n)} = \frac{c_i^{(n)}}{\sqrt{{\bs{c}^{(n)}}^T \bs{H}_{n+1} \bs{c}^{(n)}}}.
\end{align}
Together with $b_0 = \sqrt{m_0}$ and $c_{-1}^{(0)} \coloneqq 0$, the recurrence coefficients can be obtained from \eqref{eq:normalc} using \eqref{eq:ttr},
\begin{align}\label{eq:ab_aPC}
  a_n &= \frac{\bar{c}_{n-2}^{(n-1)} - b_n \bar{c}_{n-1}^{(n)}}{\bar{c}_{n-1}^{(n-1)}}, & 
  b_n &= \frac{\bar{c}_{n-1}^{(n-1)}}{\bar{c}_n^{(n)}}, &
  n &\in \N.
\end{align}

Thus, given the moments $m_i$, we first solve for the $c^{(n)}_k$ via \eqref{eq:cnk-system} and subsequently uses \eqref{eq:ab_aPC} to compute the recurrence coefficients. As with the Hankel determinant procedure in Section \ref{subsec:HD}, this procedure is susceptible to instability since the moment matrices in \eqref{eq:cnk-system} are typically unstable.

\subsection{\MC: Modified Chebyshev algorithm}\label{subsec:MC}

The previous techniques have used (monomial) moments directly and suffer from numerical stability issues. The classical Chebyshev algorithm \cite{chebyshev1859interpolation} still uses monomial moments, but it employs them through an iterative recursive approach to compute the recurrence coefficients. The technique in this section modifies the classical Chebyshev algorithm by using $\mu$-moments computed with respect to some other set of polynomials $\{q_k\}$. Typically, $q_k$ is chosen as a sequence of polynomials that are orthogonal with respect to another measure $\lambda$, where we require that the recurrence coefficients $c_n, d_n$ for $\lambda$ are known. The Modified Chebyshev algorithm is effective when $\lambda$ is chosen ``close" to $\mu$.

We define the ``mixed" moments as
\begin{align}\label{eq:mix_moments}
\sigma_{n,k} &= \int \pi_n(x) q_k(x) \dx{\mu(x)}, & n,k > -1,
\end{align}
where $\pi_n(x)$ are the monic orthogonal polynomials with respect to $\mu$. We denote $a_n, b_n$ as the recurrence coefficients of orthonormal polynomials $p_n(x)$ with respect to $\mu$. They can be used to formulate the three-term recurrence relation for monic orthogonal polynomials $\pi_n(x)$,
\begin{align}\label{eq:monic_ttr}
\pi_{n+1}(x) = (x - a_{n+1}) \pi_n(x) - b_n^2 \pi_{n-1}(x).
\end{align}

We define $c_k, d_k$ as recurrence coefficients of orthonormal polynomials $q_k(x)$. Plugging \eqref{eq:monic_ttr} into \eqref{eq:mix_moments}, the mixed moments $\sigma_{n,k}$, in turn, satisfies the recurrence relation below:
\begin{align}\label{eq:recursion_mod}
\sigma_{0,k} &= m_k, \\ \nonumber
\sigma_{n,k} &= d_k \sigma_{n-1,k-1} + (c_{k+1} - a_n) \sigma_{n-1,k} + d_{k+1} \sigma_{n-1,k+1} - b_{n-1}^2 \sigma_{n-2,k}.
\end{align}
\eqref{eq:recursion_mod} gives a routine to compute the first $N$ recurrence coefficients, which requires as input the first $2N-1$ modified moments $\{m_k\}_{k=0}^{2N-2}$ and $\{c_k, d_k\}_{k=0}^{2N-1}$.

Together with \eqref{eq:monic_ttr}, \eqref{eq:recursion_mod} and the fact that $\sigma_{-1,k} = 0$, we have the expression of the recurrence coefficients,
\begin{align}\label{eq:ab_MC}
a_1 &= c_1 + \frac{d_1 \sigma_{0,1}}{\sigma_{0,0}}, & a_n &= c_n + \frac{d_n \sigma_{n-1,n}}{\sigma_{n-1,n-1}} - \frac{d_{n-1} \sigma_{n-2,n-1}}{\sigma_{n-2,n-2}}, & n &= 2, 3, ..., \\ \nonumber
b_0 &= \sqrt{d_0 m_0}, & b_n &= \sqrt{\frac{d_{n} \sigma_{n,n}}{\sigma_{n-1,n-1}}}, & n&\in\N.
\end{align}

Given a positive measure $\mu$ on $\R$,  by choosing $\lambda$ near $\mu$ in some sense, we expect the algorithm is well, or better, conditioned \cite[Section 2.1.3]{gautschi_orthogonal_2004}.


\subsection{\SP: The Stiltjies procedure}\label{subsec:SP}

The previous procedures have used either monomial moments or general (mixed) moments with respect to a prescribed, fixed alternative basis $q_k$. In constrast, the Stieltjes procedure \cite{stieltjes1884some, gautschi_generating_1982} requires ``on-demand" computation of moments, i.e., the moments required are determined during the algorithm.  Starting with $b_0 = \left(\int \dx{\mu}\right)^{1/2}$ and $p_0(x) = 1/b_0$, $a_1$ can be computed from \eqref{eq:polymoments} with $q(x) = x p_0(x)^2$, which allows us to evaluate $p_1(x)$ by means of \eqref{eq:ttr}. $p_1(x)$, in turn. can be used to generate $b_1$. The formulae \cite[Section 2.2.3]{gautschi_orthogonal_2004} 
\begin{align}\label{eq:ab_SP}
a_n &= \int x p_{n-1}^2(x) \dx{\mu}, &
b_n &= \left(\int ( (x - a_n) p_{n-1}(x) - b_{n-1} p_{n-2}(x) )^2 \dx{\mu}\right)^{\frac{1}{2}}, &
n &\in\N,
\end{align}
for the recursion coefficients provides a natural iterative framework for computing them.
\subsection{\LZ: A Lanczos-type algorithm}\label{subsec:LZ}

We assume that the measure $\dx{\mu}$ is a discrete measure with finite support, i.e., \eqref{eq:mu} holds with $C = 0$ and $0 < M < \infty$. We wish to compute recurrence coefficients $(a_n,b_n)$ up to $n < M$, ensuring that orthogonal polynomials up to this degree exist. We could also consider applying this procedure to a finite discretization of a continuous measure; see \cite[Section 2.2.3.2 and Theorem 2.32]{gautschi_orthogonal_2004}.

The Lanczos procedure produces recurrence coefficients for the discrete measure $\mu$, and utilizes the Lanczos algorithm that unitarily triangularizes a symmetric matrix. With $(\tau_j, \nu_j)_{j=1}^M$ the quadrature rule associated to the measure $\mu$ in \eqref{eq:mu}, we define 
\begin{align*}
  \sqrt{\bs{\nu}} &\coloneqq \left( \sqrt{\nu_1}, \; \sqrt{\nu_2}, \; \ldots\; \sqrt{\nu_M} \right)^T, & \bs{D} &\coloneqq \mathrm{diag}\left( \tau_1, \; \tau_2, \; \ldots, \; \tau_M \right).
\end{align*}
We define $\bs{Q}$ as a scaled $M \times M$ Vandermonde-like matrix,
\begin{align*}
  \bs{Q} &= \mathrm{diag}\left(\sqrt{\bs{\nu}}\right) \bs{V}, & \left(\bs{V}\right)_{j,k} = p_{j-1}(\tau_k),
\end{align*}
for $j, k = 1, \ldots, M$. Then, $\bs{Q}$ is an orthogonal matrix by orthonormality of $p_n$. The orthogonality and the three-term recurrence further imply that,
\begin{align*}
 \left(\begin{array}{cc} 1 & \bs{0}^T \\ \bs{0} & \bs{Q} \end{array}\right)
   \left(\begin{array}{cc} 1 & \sqrt{\bs{\nu}}^T \\ \sqrt{\bs{\nu}} & \bs{D} \end{array}\right)
 \left(\begin{array}{cc} 1 & \bs{0}^T \\ \bs{0} & \bs{Q}^T \end{array}\right)
    = 
  \left(\begin{array}{cc} 1 & b_0 \bs{e}_1^T \\ b_0 \bs{e}_1 & \bs{J}_M(\mu) \end{array}\right),
\end{align*}
where $\bs{e}_1 = (1, 0, 0, \ldots)^T \in \R^{M}$. The Lanczos algorithm, given the middle matrix on the left-hand side, computes the unitary triangularization above and outputs the right-hand side, which identifies the Jacobi matrix $\bs{J}_M$ in \eqref{eq:jacobi-matrix}, and, hence, the recurrence coefficients. See \cite[Section 2.2.3.2]{gautschi_orthogonal_2004} for more details. It is well known that the standard Lanczos algortihm is numerically unstable, so that stabilization procedures must be employed \cite{rutishauser1963jacobi,gragg1984numerically}. We use a ``double orthogonalization" stabilization technique to avoid instability. Our results suggest that, for discrete measures, this procedure is more accurate than all the alternatives, see Section \ref{subsec:discrete_cheb}.

\section{\PCL: A hybrid predictor-corrector Lanczos procedure}\label{sec:hybrid}

The main goal of this section is to describe a procedure by which we compute recurrence coefficients for $\mu$ of the form \eqref{eq:mu}. The procedure entails knowledge of the continuous weights $\{w_j\}_{j=1}^C$ and their respective supporting intervals, $\{I_j\}_{j=1}^C$, along with the discrete part of the measure encoded by the nodes and weights $\left(\tau_j, \nu_j\right)_{j=1}^M$. In section \ref{subsec:moments}, we will also utilize the singularity behavior of the weights $w_j$ dictated by the constants $\alpha_j$ and $\beta_j$ in \eqref{eq:singularity-behavior} to compute moments. 

Section \ref{subsec:PC} first introduces a new procedure to compute recurrence coefficients for a measure with a continuous density using polynomial moments. Section \ref{subsec:moments} then discusses our particular strategy for computing these moments. Finally, section \ref{subsec:PCL} introduces a procedure based on the multiple component approach in \cite{gautschi_generating_1982} for computing recurrence coefficients for a measure of general form \eqref{eq:mu}.

\subsection{\PC: Predictor-corrector method}\label{subsec:PC}
In this section, we describe a Stieltjes-like procedure for computing recurrence coefficients. Although this works for general measures, we are mainly interested in applying this technique for measures $\mu$ that have a continuous density. The high-level algorithm, like the previous ones we have discussed, is iterative. Suppose for some $n \geq 0$ we know the coefficient tableau,
\begin{align*}
  \begin{array}{cccccc}
                & a_1(\mu) & a_2(\mu) & \cdots & a_{n}(\mu) \\ 
       b_0(\mu) & b_1(\mu) & b_2(\mu) & \cdots & b_{n}(\mu).
  \end{array}
\end{align*}
These coefficients, via \eqref{eq:ttr}, define $p_0, \ldots, p_n$ that are orthonormal under a $\dx{\mu}$-weighted intergral. In order to compute $a_{n+1}$ and $b_{n+1}$, we make educated guesses for these coefficients, and correct them using computed moments. The procedure is mathematically equivalent to the Stieltjes procedure:
We define a new set of recurrence coefficients $\{\widetilde{a}_j, \widetilde{b}_j\}_{j=0}^{n+1}$, where
\begin{subequations}
\begin{align}\label{eq:coeff-ansatz}
  \widetilde{a}_{j} &= a_{j}, & \widetilde{b}_{j} &= b_{j}, \hskip 10pt j = 0, \ldots, n, \\
  \widetilde{a}_{n+1} &= a_{n}, & \widetilde{b}_{n+1} &= b_{n}, 
\end{align}
In particular, corrections $\Delta a_{n+1} \in \R$ and $\Delta b_{n+1} > 0$ exist such that 
\begin{align}\label{eq:coeff-diff}
  a_{n+1} &= \widetilde{a}_{n+1} + \Delta a_{n+1}, & 
  b_{n+1} &= \widetilde{b}_{n+1} \Delta b_{n+1}.
\end{align}
\end{subequations}
Our procedure will compute the corrections $\Delta a_{n+1}$ and $\Delta b_{n+1}$. The tableau of coefficients $\widetilde{a}_{n+1}$ and $\widetilde{b}_{n+1}$ 
\begin{align*}
  \begin{array}{ccccc}
            & a_1(\mu) & \cdots & a_{n}(\mu) & \widetilde{a}_{n+1}(\mu) \\ 
   b_0(\mu) & b_1(\mu) & \cdots & b_{n}(\mu) & \widetilde{b}_{n+1}(\mu),
  \end{array}
\end{align*}
can be used with \eqref{eq:ttr} to generate the polynomials $p_0, \ldots, p_n$, along with $\widetilde{p}_{n+1}$, defined as
\begin{align}\label{eq:ptilde}
  \widetilde{b}_{n+1} \widetilde{p}_{n+1} &\coloneqq (x - \widetilde{a}_{n+1}) p_n - b_n p_{n-1}.
\end{align}
Since $\widetilde{p}_{n+1}$ and $p_{n+1}$ were generated using the same coefficients $(a_j, b_j)$ up to index $j = n$, then they are both orthogonal to all polynomials of degree $n-1$ or less. However, $\widetilde{p}_{n+1}$ is not orthogonal to $p_n$ in general. We can choose $\Delta a_{n+1}$ to enforce this orthogonality, which requires computing a polynomial moment.

Once $a_{n+1} = \widetilde{a}_{n+1} + \Delta a_{n+1}$ is successfully computed, we can similarly define another degree-$(n+1)$ polynomial $\widehat{p}_{n+1}$ through the relation,
\begin{align}\label{eq:phat}
  \widetilde{b}_{n+1} \hat{p}_{n+1} &\coloneqq (x - a_{n+1}) p_n - b_n p_{n-1}.
\end{align}
This polynomial differs from $p_{n+1}$ by only a multiplicative constant, which can again be determined through a moment computation and used to compute $\Delta b_{n+1}$. We formalize the discussion above through the following result:
\begin{lemma}\label{lemma:Gcorrections}
  With $\widetilde{p}_{n+1}$ and $\hat{p}_{n+1}$ defined as in \eqref{eq:ptilde} and \eqref{eq:phat}, respectively, let
  \begin{subequations}\label{eq:G-entries}
  \begin{align}
    \label{eq:G-offdiag}    G_{n,n+1} &\coloneqq \int_{\R} p_n(x) \widetilde{p}_{n+1} (x) \dx{\mu}(x), \\
    \label{eq:G-diag}    G_{n+1,n+1} &\coloneqq \int_{\R} \hat{p}^2_{n+1} (x) \dx{\mu}(x),
  \end{align}
  \end{subequations}
  Then,
  \begin{align}\label{eq:Delta-coeffs}
    \Delta a_{n+1} &= G_{n,n+1} b_n, & \Delta b_{n+1} &= \sqrt{ G_{n+1,n+1}}.
  \end{align}
\end{lemma}

\begin{proof}
  Starting from the definition \eqref{eq:ptilde} for $\widetilde{p}_{n+1}$, we replace $x p_n$ with the right-hand side of \eqref{eq:ttr}, yielding,
  \begin{align}\nonumber
    \widetilde{p}_{n+1} &= \Delta b_{n+1} \left[ \frac{1}{b_{n+1}} \left( x - a_{n+1} \right) p_n - b_n p_{n-1} + \Delta a_{n+1} \frac{1}{b_{n+1}} p_n\right] \\\label{eq:lemma-temp-1}
    &= \Delta b_{n+1} p_{n+1} + \frac{\Delta a_{n+1} \Delta b_{n+1}}{b_{n+1}} p_n
  \end{align}
  Thus, due to orthogonality of $\{p_j\}_{j \geq 0}$, we have
  \begin{align*}
    G_{n, n+1} &= \int p_n(x) \widetilde{p}_{n+1}(x) \dx{\mu}(x) \stackrel{\eqref{eq:lemma-temp-1}}{=} \frac{\Delta a_{n+1} \Delta b_{n+1}}{b_{n+1}} \stackrel{\eqref{eq:coeff-diff}}{=} \frac{\Delta a_{n+1}}{b_n},
  \end{align*}
  which shows the first relation in \eqref{eq:Delta-coeffs}. To show the second relation, first we combine \eqref{eq:ttr} and \eqref{eq:phat} to show,
  \begin{align*}
    \widetilde{b}_{n+1} \hat{p}_{n+1}(x) = (x - a_{n+1}) p_n - b_n p_{n-1} = b_{n+1} p_{n+1},
  \end{align*}
  so that 
  \begin{align*}
    G_{n+1,n+1} = \int \hat{p}_{n+1}^2(x) \dx{\mu}(x) = \left(\frac{b_{n+1}}{\widetilde{b}_{n+1}}\right)^2 \int p_{n+1}^2(x) \dx{\mu}(x) = (\Delta b_{n+1})^2,
  \end{align*}
  proving the second relation.
\end{proof}

The results \eqref{eq:Delta-coeffs} and \eqref{eq:Delta-coeffs} are the proposed approach: The moments $G_{n,n+1}$ and $G_{n+1,n+1}$ in \eqref{eq:G-offdiag} and \eqref{eq:G-diag} are polynomial moments that can be computed. We can subsequently use \eqref{eq:Delta-coeffs} and \eqref{eq:coeff-diff} to compute the desired $a_{n+1}$ and $b_{n+1}$.

The methodology of this section can then be iterated in order to compute as many recurrence coefficients $a_n$ and $b_n$ as desired. However, we must compute the $G_{n,n+1}$ and $G_{n+1,n+1}$ coefficients (which are similar to the moments required by the Stieltjes procedure). The main difference in our algorithm is that we use moments to compute $a_{n+1} - a_n$ and $b_{n+1}/b_n$ that are typically close to 0 and 1, respectively, instead of simply $a_n$ and $b_n$, which in general can be arbitrarily small or large numbers. We next summarize one particular strategy for computing these moments assuming that a type of characterization of $\mu$ is available. 

\subsection{Computation of polynomial moments}\label{subsec:moments}

The previous section shows that we can compute recurrence coefficients for the measure $\mu$ if we can compute some of its moments, in particular $G_{n,n+1}$ and $G_{n+1,n+1}$. We briefly describe in this section how we compute moments for measures of the form \eqref{eq:mu} with knowledge of the singularity behavior in \eqref{eq:singularity-behavior}. The moment of a polynomial $q$ for $\mu$ can be written as
\begin{align*}
  \int q(x) \dx{\mu}(x) = \sum_{j=1}^C \int_{I_j} q(x) w_j(x) \dx{x} + \sum_{j=1}^M \nu_j q(\tau_j),
\end{align*}
so that the only difficult part is to compute $\int_{I_j} q(x) w_j(x) \dx{x}$ for each $j$. 

Suppose first that $I_j$ is compact, i.e., that $I_j = [\ell, r]$ for finite $\ell, r$. Then we rewrite the integral as
\begin{align*}
  \int_{I_j} q(x) w_j(x) \dx{x} &= \frac{r-\ell}{2} \int_{-1}^1 q(A(u)) w_j(A(u)) \dx{u}, & A(u) \coloneqq \left(\frac{r-\ell}{2}\right) u + \frac{r + \ell}{2}.
\end{align*}
$w_j$ obeying the limiting conditions \eqref{eq:singularity-behavior} with constants $\alpha_j,\beta_j$ implies that $w_j(A(u))$ behaves like $(1 - u^{\alpha_j})$ near $u = 1$, and like $(1 + u^{\beta_j})$ near $u = -1$. When $\alpha_j = \beta_j = 0$, then a global $\dx{x}$-Gaussian quadrature rule will be efficient in evaluating this integral, but the accuracy will suffer when either constant differs from 0. To address this problem, we can further rewrite the integral as:
\begin{align*}
  \int_{I_j} q(x) w_j(x) \dx{x} = \frac{r-\ell}{2} \int_{-1}^1 q(A(u)) \omega_j(u) \dx{\mu}^{(\alpha_j,\beta_j)}(u),
\end{align*}
where $\mu^{(\alpha_j,\beta_j)}$ is a Jacobi measure on $[-1,1]$, and $\omega_j$ is $w_j$ multiplied by the appropriate factors,
\begin{align*}
  \dx{\mu}^{(\alpha_j,\beta_j)}(u) &= (1-u)^{\alpha_j} (1+u)^{\beta_j} \dx{x}, &
  \omega_j(u) &\coloneqq w_j(A(u)) (1-u)^{-\alpha_j} (1+u)^{-\beta_j}.
\end{align*}
The advantage of this formulation is that $\omega_j$ is now smooth at the boundaries $u = \pm 1$, and if in addition it is smooth on the interior of $[-1,1]$, then a Jacobi $(\alpha_j,\beta_j)$-Gaussian quadrature rule will efficiently evaluate the integral. Therefore, if $\left(u_k, \lambda_k\right)_{k=1}^K$ is a $K$-point Jacobi $(\alpha_j, \beta_j)$-Gaussian quadrature rule, we approximate the integral as
\begin{align*}
  \int_{I_j} q(x) w_j(x) \dx{x} \approx \sum_{k=1}^K \lambda_k \omega_j(u_k) q(A(u_k)),
\end{align*}
where the nodes and weights can be computed through the spectrum of $J_K(\mu^{(\alpha_j, \beta_j)})$ since the recurrence coefficients of these measures are explicitly known. In particular, all the quadrature nodes $u_k$ lie interior to $[-1,1]$, so that the above procedure does \textit{not} require evaluation of $\omega_j$ at $u = \pm 1$. We adaptively choose $K$, i.e., increasing $K$ until the difference between approximations is sufficiently small.

\subsection{\PCL: A hybrid Predictor-corrector Lanczos method}\label{subsec:PCL}

The full procedure we describe in this section combines the strategies in Sections \ref{subsec:PC} and \ref{subsec:moments}, along with the (stabilized) Lanczos procedure in Section \ref{subsec:LZ}. Assuming that we \textit{a proiri} know that the first $N$ recurrence coefficients $\{a_n, b_n\}_{n=0}^{N-1}$ are required for $\mu$, then the main idea here is to construct a fully discrete measure $\nu$ whose moments up to degree $2N-2$ match those of $\mu$.

We accomplish this as follows: Recall that the continuous densities $\{w_j\}_{j=1}^C$ of the measure $\mu$ in \eqref{eq:mu} are known, along with their boundary singularity behavior in \eqref{eq:singularity-behavior}. Then for each $j$, the \PC{} procedure in sections \ref{subsec:PC} and \ref{subsec:moments} can be used to compute the first $N+1$ recurrence coefficients for $w_j$, $\{a_{j,n}, b_{j,n}\}_{n=0}^{N}$. Using these recurrence coefficients, an $N$-point Gaussian quadrature rule $(x_{j,k}, \lambda_{j,k})_{k=1}^N$ can be computed that exactly integrates all polynomials up to degree $2N-1$ with respect to the weight $w_j$:
\begin{align*}
  \int_{I_j} q(x) w_j(x) \dx{x} &= \sum_{k=1}^N \lambda_{j,k} q\left(x_{j,k}\right), & \deg q &\leq 2N-1.
\end{align*}
After this quadrature rule is computed for every $j = 1, \ldots, C$, the discrete measure $\nu$, defined as 
\begin{align}\label{eq:nu}
  \nu \coloneqq \sum_{j=1}^C \sum_{k=1}^N \lambda_{j,k} \delta_{x_{j,k}} + \sum_{j=1}^M \nu_j \delta_{\tau_j},
\end{align}
and has moments that match those of $\mu$ up to degree $2N-1$. Once this procedure is completed, we employ the Lanczos procedure in Section \ref{subsec:LZ} to compute the first $N$ recurrence coefficients for $\nu$, which equal those for $\mu$. The main reason we employ the Lanczos scheme (as opposed to any other approach) is that, for discrete measures, the Lanczos procedure appears more empirically stable than all other procedures we consider, cf. Section \ref{subsec:multi}.

Note that if $C = 1$ and $M=0$, then the Lanczos procedure is not needed at all since $(a_{1,n}, b_{1,n})_{n=0}^{N-1}$ are the desired coefficients, and if $C = 0$, then only the Lanczos procedure need be queried since no quadrature is required.

The above is essentially a complete description of the \PCL{} algorithm. However, we include one additional adaptive procedure to ensure correct computation of the moments. Let $\{N_s\}_{s \geq 0}$ be an increasing sequence of positive integers. A strategy for determining the sequence of $N_s$ can be found in \cite{gautschi1994algorithm, gautschi_orthogonal_2004},
\begin{align*}
N_0 &= N, & N_s &= N_{s-1} + \Delta_s, & s&= 1, 2, \dots, \\
\Delta_1 &= 1, & \Delta_s &= 2^{\lfloor\frac{s}{5}\rfloor} N, & s &= 2, 3, \dots \ .
\end{align*}
We define $\nu_s$ as the measure \eqref{eq:nu} with $N \gets N_s$. We use \PCL{} to compute numerical approximations $\{a_n^{[s]}, b_n^{[s]} \}_{n \geq 0}$ to the recurrence coefficients for $\nu_s$. (I.e., we use \PC{} to compute the $N_s$-point quadrature rule $(x_{j,k}, \lambda_{j,k})_{k=1}^{N_s}$ and subsequently use \LZ{} to compute the recurrence coefficients for $\nu_s$.) With the (approximate) coefficients for $\nu_s$ and $\nu_{s-1}$, if the condition
\begin{align*}
  \left|b_n^{[s]} - b_n^{[s-1]} \right| &\le \epsilon |b_n^{[s]}|, & n &= 0, 1, \dots, N-1.
\end{align*}
is satisfied, then we return the computed coefficients for $\nu_s$. Otherwise, we set $s \gets s+1$ and test the condition above again. This adaptive procedure is similar to those employed in \cite{gautschi1994algorithm, gautschi_orthogonal_2004}. In our computations we set $\epsilon = 10^{-12}$, and we set an upper limit of $N_s$ as $N_s^{\max} = 10 N$ for all $s$, which will usually be satisfactory.
\section{Numerical Experiments}\label{sec:numerical}

We now present numerical examples to illustrate the performance of our algorithm by computing the first $N$ three-term recurrence coefficients for different types of measures $\mu$. Our results will consider all the algorithms in Table \ref{tab:notation}: the first six in section \ref{sec:methods} and the last two new procedures proposed in Section \ref{sec:hybrid}. We implement all the algorithms in Python. All the computations are carried out on a MacBook Pro laptop with a 3.1 GHz Intel(R) Core(TM) i5 processor and 8 GB of RAM.

Examples can be classified according to whether we have a way to compute the exact recurrence coefficients. When this is the case, we define $\{\hat{a}_n, \hat{b}_n\}_{n=0}^{N-1}$ as the first $N$ exact coefficients and $\{a_n, b_n\}_{n=0}^{N-1}$ as coefficients that are computed from any particular algorithm. The error $e_N$ can be denoted by an $\ell^2$-type norm,
\begin{equation}\label{eq:error}
e_N = \left(\sum_{n=0}^{N-1} \left[\left(a_n - \hat{a}_n\right)^2 + \left(b_n - \hat{b}_n\right)^2\right] \right)^{\frac{1}{2}}.
\end{equation}
If the exact coefficients are not available, we consider another error metric. If $\{p_n(x)\}_{n=0}^{N-1}$ is a polynomial basis produced through the three-term recurrence \eqref{eq:ttr} using the computed coefficients by $\{a_n, b_n\}_{n=0}^{N-1}$, then let $\bs{A}$ be an $N \times N$ matrix with entries
\begin{align*}
  \bs{A}_{m,n} &= \int_{\R} p_{n-1}(x)p_{m-1}(x) \dx{\mu}(x), & n, m &= 1, \ldots, N,
\end{align*}
which equals $\delta_{n,m}$ if $\hat{a}_n = a_n$ and $\hat{b}_n = b_n$. The new error indicator $f_N$ we compute is
\begin{equation}\label{eq:newerror}
  f_N = \left\| \bs{A} - \bs{I} \right\|_F,
\end{equation}
where $\|\cdot\|_F$ is the Frobenius norm on matrices and $\bs{I}$ is the $N \times N$ identity matrix. 

The computational timing results that measure efficiency are averaged over 100 runs of any particular algorithm.

\subsection{Freud weights}\label{subsec:freud}

One computational strategy for determining the recurrence coefficients for Freud weights of the form \eqref{eq:freud} on the entire real line is to use the (``non-modified") Chebyshev algorithm, which requires monomial moments and employs a recurrence similar to \eqref{eq:recursion_mod}. The monomial moments of \eqref{eq:freud} are explicitly computable as simple evaluations of the Euler Gamma function, but numerical instabilities typically develop in such an approach due to roundoff error; to combat this limitation, computations may be completed in variable precision arithmetic, resulting in a procedure that correctly computes the recurrence coefficients \cite{2009Variable}. In this section, we use this VPA procedure to generate recurrence coefficients treated as ``exact" for use in computing errors. In particular, we employ the \texttt{sr\_freud.m} routine from \cite{gautschivpa} that utilizes variable-precision arithmetic in Matlab \cite{mathworksvpa}. 


\begin{figure}[h]
  \centering
  \includegraphics[width=0.8\textwidth]{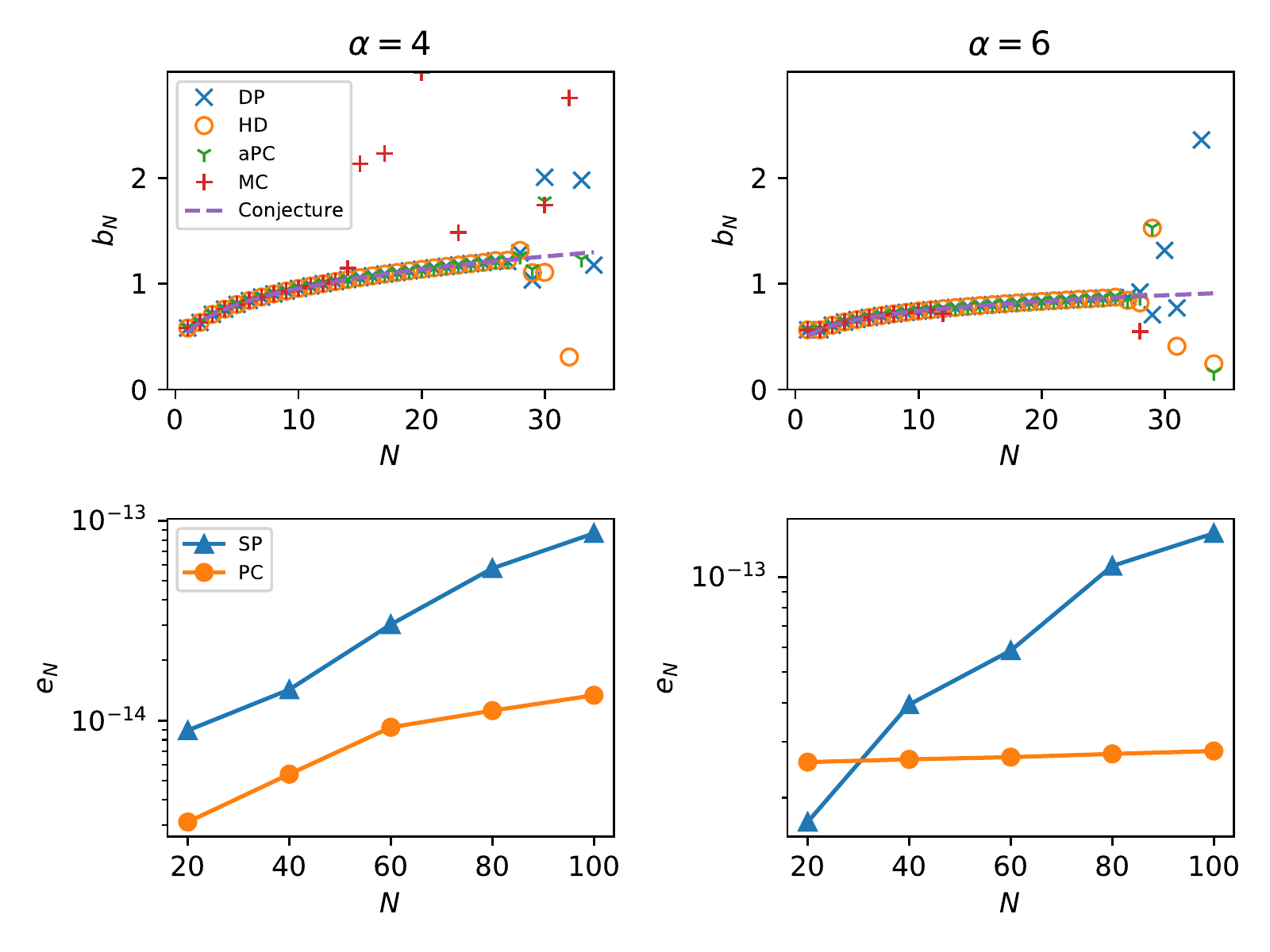}
  \caption{Example for Section \ref{subsec:DP}: the top two plots are recursion coefficients $b_N$ computed by \DP, \HD, \aPC{}, \MC{} and Freud conjecture in \eqref{eq:conjecture_freud4} and \eqref{eq:conjecture_freud6}. The two plots at the bottom show errors $e_N$ of \SP{} and \PC{}}
  \label{fig:freud}
\end{figure}

We compute recurrence coefficients using the \DP, \HD, \aPC, and \MC{} methods for Freud exponents $\alpha = 4, 6$. The \DP{} recursion for each of the two cases is simple, given by \eqref{eq:recursion_freud4} and \eqref{eq:recursion_freud6}, respectively. For the \MC{} method, we use Hermite orthogonal family for $q_k$ in \eqref{eq:mix_moments} that is orthogonal with respect to $\lambda$. The top two plots in Figure \ref{fig:freud} show that each of these methods is not computationally useful since instabilities develop quickly. In contrast, both the \SP{} and \PC{} approaches can effectively compute recurrence coefficients, which we show in the bottom two plots of Figure \ref{fig:freud}. In terms of efficiency, Table \ref{tab:freud_time}  illustrates that the ``exact" VPA procedure is several orders of magnitude more expensive than all other approaches, and that \SP{} and \PC{} are competitive. Code that reproduces this example is available in the routine \texttt{ex\_freud\_4.py} and \texttt{ex\_freud\_6.py} from \cite{Univariate_ttr_examples}.

\begin{table}
  \begin{center}
  \resizebox{0.95\textwidth}{!}{
    \renewcommand{\tabcolsep}{0.4cm}
    \renewcommand{\arraystretch}{1.3}
    {\scriptsize
      \begin{tabular}{c|cc|cc|cc|cc|cc}
      \toprule
        Method & \multicolumn{2}{c}{$N = 20$} & \multicolumn{2}{c}{$N =40$} & \multicolumn{2}{c}{$N = 60$} & \multicolumn{2}{c}{$N = 80$} & \multicolumn{2}{c}{$N = 100$} \\\midrule
         VPA & 18.86 & 19.13 & 99.38 & 101.10 & 293.44 & 300.41 & 631.20 & 633.20 & 1196.29 & 1362.86 \\
         \SP & 0.24 & 0.21 & 0.75 & 0.63 & 1.60 & 1.32 & 2.72 & 2.24 & 4.12 & 3.42 \\
         \PC & 0.25 & 0.22 & 0.75 & 0.65 & 1.60 & 1.34 & 2.72 & 2.27 & 4.12 & 3.40 \\
      \bottomrule
      \end{tabular}
    }
 }
  \end{center}
  \caption{Example for Section \ref{subsec:DP}: elapsed time (s) for Freud weight when $\alpha = 4$ (subcolumns on the left) and $\alpha = 6$ (subcolumns on the right).}\label{tab:freud_time}
\end{table}

\subsection{Piecewise smooth weight}\label{subsec:pws}

We consider the measure $\dx{\mu}(x) = w(x) \dx{x}$ on $[-1,1]$, where
\begin{align*}
\omega(x) = 
\begin{cases}
	\mid x\mid^\gamma (x^2 - \xi^2)^p (1 - t^2)^q, & \text{$x \in [-1, -\xi] \cup [\xi, 1]$}\\ 0, elsewhere, &\text{$0<\xi<1, p>-1, q>-1, \gamma \in \R$}.
\end{cases}
\end{align*}

For certain choices of $\gamma, p, q$, there is theory regarding the resulting orthogonal polynomials \cite{barkov1960some}, and such weights arise in applications \cite{wheeler1984modified}. In the special cases $\gamma = \pm1, p = q = \pm 1/2$, closed-form representations for the recurrence coefficients can be computed \cite{gautschi1984some}. For example, the exact formula for the recurrence coefficients for the case $\gamma = 1, p = q = -1/2, \eta = (1 - \xi)/(1 + \xi)$ is given by
\begin{align*}
\hat{b}_0 &= \sqrt{\pi}, & \hat{b}_1 &= \sqrt{\frac{1 + \xi^2}{2}}, \\
\hat{b}_{2n} &= \sqrt{\frac{(1 - \xi^2) (1 + \eta^{2n-2})}{4(1 + \eta^{2n})}}, & \hat{b}_{2n+1} &= \sqrt{\frac{(1 + \xi^2) (1 + \eta^{2n+2})}{4(1 + \eta^{2n})}}, & n \in \N,
\end{align*}
with $\hat{a}_n = 0$ for all $n$.

\begin{table}
  \begin{center}
  \resizebox{0.95\textwidth}{!}{
    \renewcommand{\tabcolsep}{0.4cm}
    \renewcommand{\arraystretch}{1.3}
    {\scriptsize
      \begin{tabular}{c|cc|cc|cc|cc|cc}
      \toprule
        Method & \multicolumn{2}{c}{$N = 20$} & \multicolumn{2}{c}{$N = 40$} & \multicolumn{2}{c}{$N = 60$} & \multicolumn{2}{c}{$N = 80$} & \multicolumn{2}{c}{$N = 100$} \\\midrule
         \HD & 6.05e-02 & 0.003 & --- & --- & --- & --- & --- & --- & --- & --- \\
         \aPC & 6.05e-02 & 0.001 & --- & --- & --- & --- & --- & --- & --- & --- \\
         \MC & 2.34e-15 & 0.001 & 1.00e+00 & 0.006 & --- & --- & --- & --- & --- & --- \\
         \SP & 4.73e-14 & 0.10 & 2.85e-13 & 0.28 & 3.85e-13 & 0.57 & 3.99e-13 & 0.93 & 4.62e-13 & 1.39 \\
         \PC & 9.08e-15 & 0.10 & 1.80e-14 & 0.29 & 3.13e-14 & 0.57 & 5.14e-14 & 0.94 & 7.27e-14 & 1.40 \\
      \bottomrule
      \end{tabular}
    }
 }
  \end{center}
  \caption{Example for Section \ref{subsec:pws}: errors $e_N$ (subcolumns on the left) and elapsed time (s) (subcolumns on the right) when $\gamma = 1, p = q = -1/2$. Here --- means a NaN value due to the numerical overflow from the instability of the corresponding method}\label{tab:pws}
\end{table}

A Legendre orthogonal family for $q_k$  in \eqref{eq:mix_moments} that is orthogonal with respect to $\lambda$ is chosen for the \MC{} method. For the choice $\gamma = 1, p = q = -1/2$ and $\xi = 1/10$, Table \ref{tab:pws} illustrates the accuracy and cost of the algorithms \HD{}, \aPC{}, \MC{}, \SP{}, and \PC{}. We observe that only the \SP{} and \PC{} approaches yield reasonable accuracy, with \PC{} being slightly more accurate. We omit results for other choices of $(\gamma, p, q)$, which produce nearly identical results. The results from this table can be produced from \texttt{ex\_pws.py} in \cite{Univariate_ttr_examples}.

\subsection{Transformed discrete Chebyshev}\label{subsec:discrete_cheb}

In the previous example, we compute the recurrence coefficients of ``continuous" orthogonal polynomials with respect to $\mu$ on bounded or unbounded supports. We now consider the support of $\mu$ that consists of a discrete set of points.

Given a positive number $M$, we define the nodes $\tau_j = (j-1)/M$ and $\nu_j = 1/M$ for $j = 1, 2, \dots, M$. Then,  the transformed discrete Chebyshev \cite[Example 2.26]{gautschi_orthogonal_2004} measure is given as
\begin{align*}
\dx{\mu}(x) &= \sum_{j=1}^M \frac{1}{M} \delta_{\frac{j-1}{M}} \dx{x}, & j = 1, 2, \dots, M,
\end{align*}
i.e., an equally spaced and equally weighted discrete measure on $[0,1)$. The recurrence coefficients are known explicitly if a linear transformation of variables is applied to the discrete Chebyshev measure with canonical support points \cite[Section 1.5.2]{gautschi_orthogonal_2004}. For a given size of supports, $M$, with $\hat{b}_0 = 1$,
\begin{align*}
\hat{a}_n &= \frac{M-1}{2M}, & \hat{b}_n &= \sqrt{\frac{1 - {(\frac{n}{M})}^2}{4(4 - {(\frac{n}{M})}^2)}}, & n = 1, 2, \dots, M-1.
\end{align*}

\begin{figure}[h]
  \centering
  \includegraphics[width=0.8\textwidth]{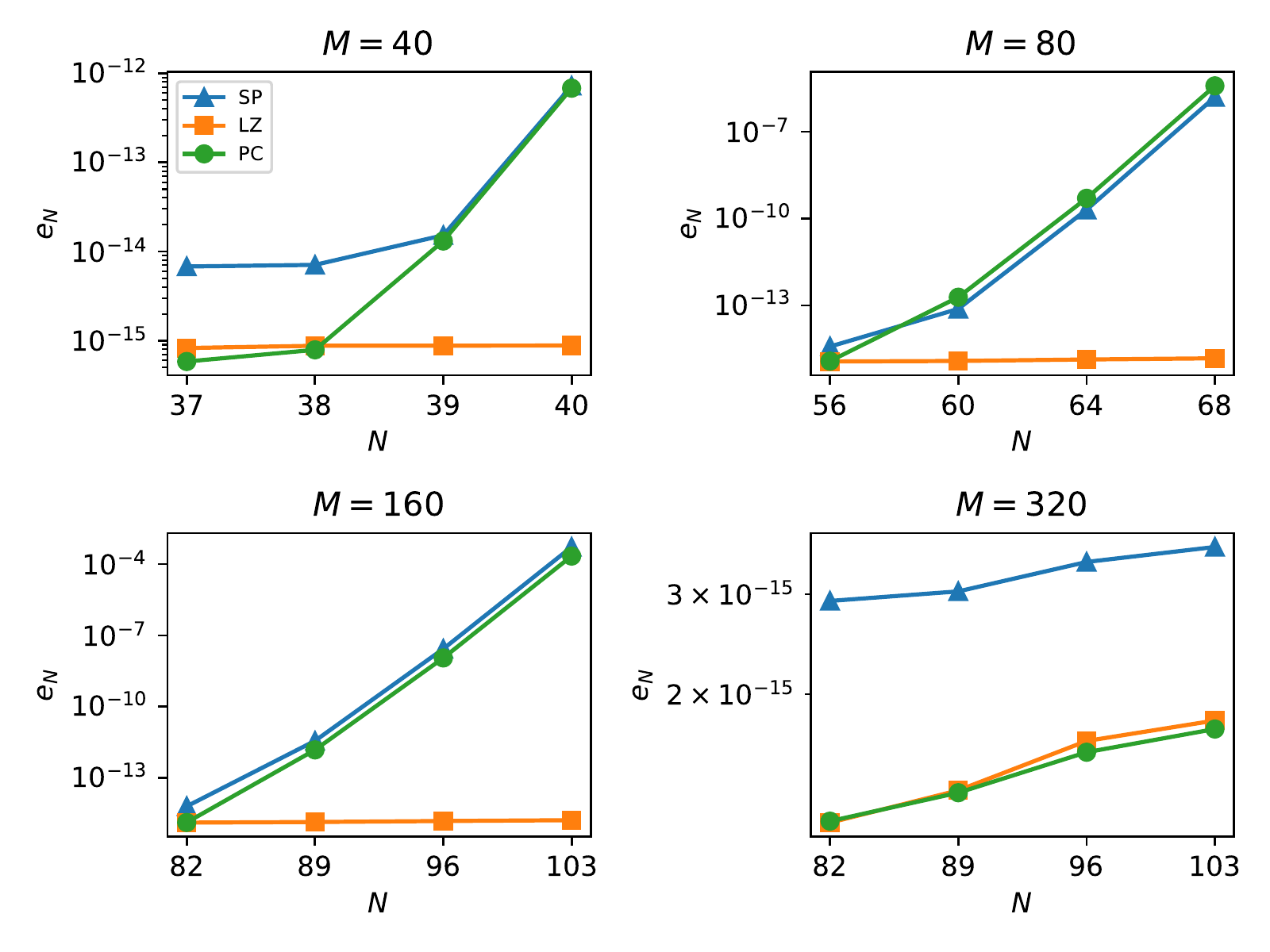}
  \caption{Example for Section \ref{subsec:discrete_cheb}: the first three plots compute errors $e_N$ for different $N$ portion of distinct $M$ and the last two plots for the same $N$ but for distinct $M$.}
  \label{fig:discrete_cheb}
\end{figure}
In Figure \ref{fig:discrete_cheb}, the methods \HD{}, \aPC{} and \MC{} are omitted since their instabilities develop very quickly. An NaN value appears when the required number of recurrence coefficients, $N$, is less than $20$. We compare the \SP{}, \LZ{} and \PC{} approaches on measure support sizes $M = 40, 80, 160, 320$. We observe that the \LZ{} approach is effective for all choices of $M$, and when $N$ is comparable to $M$, the \SP{} and \PC{} approaches become inaccurate. The lower two plots of Figure \ref{fig:discrete_cheb} show that when $M$ is notably larger than $N$, all three approaches produce good results. In particularly, all the numerical results in this subsection are are produced by \texttt{ex\_discrete\_cheb.py} in \cite{Univariate_ttr_examples}.

\subsection{Discrete probability density function}\label{subsec:discrete_convolution}

High-dimensional integration is a common problem in scientific computing arising from, for example, the need to estimate expectations in uncertainty quantification \cite{smith2013uncertainty, sullivan2015introduction}. 
Many integrands for such integrals found in scientific computing applications map a large number of input variables to an output quantity of interest, but admit low-dimensional ridge structure that can be exploited to accelerate integration. A ridge function \cite{pinkus2015ridge} is a function $f: \R^m \rightarrow \R$ of the form
\begin{align*}
f(\bs{x}) = g(\bs{a}^T \bs{x}),
\end{align*}
where $\bs{a} \in \R^m$ is a constant vector called the ridge direction and $g: \R \rightarrow \R$ is the ridge profile. For such functions, we clearly have that $f$ depends only on a scalar variable $y \coloneqq \bs{a}^T \bs{x}$. In applications, we frequently wish to integrate $f$ with respect to some $m$-dimensional probability measure $\rho$ on $\bs{x}$,  which can be simplified by integrating over the scalar variable $y$ with respect to the univariate measure $\mu$ that is the push-forward of $\rho$ under the map $\bs{x} \mapsto \bs{a}^T \bs{x}$. Thus, the goal is to compute recurrence coefficients for $\mu$.

In practice the multivariate measure $\rho$ is known, but computing the univariate measure $\mu$ exactly is typically not feasible. However, an approximation to $\mu$ can be furnished using the procedure in \cite[Section 2.2]{glaws2019gaussian} that randomly generates $M$ i.i.d. samples $\{\bs{x}_j\}_{j=1}^M$ from $\rho$, and defines $\mu$ as a discrete measure supported on the projection of these samples onto the real line:
\begin{align*}
  \dx{\mu}(x) &= \sum_{j=1}^M \frac{1}{M} \delta_{\tau_j} \dx{x}, & \tau_j &\coloneqq \bs{a}^T \bs{x}_j.
\end{align*}
To compute quadrature rules with respect to this measure, we take $\rho$ as the uniform measure on the $m$-dimensional hypercube $[-1,1]^m$. Let $m = 25$, and $\bs{a} \in \R^{25}$ is chosen randomly. We then test for $M = 100, 300$.


\begin{table}
  \begin{center}
  \resizebox{0.95\textwidth}{!}{
    \renewcommand{\tabcolsep}{0.4cm}
    \renewcommand{\arraystretch}{1.3}
    {\scriptsize
      \begin{tabular}{c|cc|cc|cc|cc|cc}
      \toprule
        Method & \multicolumn{2}{c}{$N = 20$} & \multicolumn{2}{c}{$N =40$} & \multicolumn{2}{c}{$N = 60$} & \multicolumn{2}{c}{$N = 80$} & \multicolumn{2}{c}{$N = 100$} \\\midrule
         \HD & 1.69e-07 & 1.24e-07 & --- & --- & --- & --- & --- & --- & --- & --- \\
         \aPC & 7.72e-08 & 3.80e-08 & 2.03e+05 & 7.67e+05 & 1.35e+27 & 3.90e+25 & 5.85e+47 & 2.71e+55 & 4.75e+67 & 9.70e+72 \\
         \MC & 3.02e-09 & 3.60e-09 & --- & --- & --- & --- & --- & --- & --- & --- \\
         \SP & 2.74e-15 & 3.39e-15 & 9.26e-15 & 8.53e-15 & 5.94e-10 & 2.62e-14 & 4.00e+00 & 5.20e-14 & 7.48e+00 & 9.37e-14 \\
         \LZ & 4.75e-15 & 3.87e-15 & 2.95e-14 & 1.10e-14 & 3.45e-09 & 1.73e-14 & 1.86e+68 & 3.38e-14 & 2.50e+68 & 9.29e-14 \\
         \PC & 3.96e-15 & 4.54e-15 & 1.17e-14 & 9.57e-15 & 1,03e-09 & 1.50e-14 & 4.00e+00 & 2.47e-12 & 7.48e+00 & 1.41e-13 \\
      \bottomrule
      \end{tabular}
    }
 }
  \end{center}
  \caption{Example for Section \ref{subsec:discrete_convolution}: errors $f_N$ when $M = 100$ (subcolumns on the left) and $M = 300$ (subcolumns on the right). Here --- means a NaN value due to the numerical overflow from the instability of the corresponding method.}\label{tab:discrete_convolution}
\end{table}

Since we do not have an expression for the exact recurrence coefficients, we measure errors using the metric $f_N$ in \eqref{eq:newerror}.
As shown in Table \ref{tab:discrete_convolution}, the computed recursion coefficients are not as accurate when $N$ is closer to $M$, no matter what method is used. However, the methods \SP, \LZ{} and \PC{} all perform better when $M$ is large enough.
Code that reproduces this example is available in the routine \texttt{ex\_discrete\_convolution.py} in \cite{Univariate_ttr_examples}. 

\subsection{Multiple component: Chebyshev weight function plus a discrete measure}\label{subsec:multi}

The measure to be considered is the normalized Jacobi weight function on
$[-1,1]$ with a discrete $M$-point measure added to it,
\begin{align}\label{eq:multiple} 
\dx{\mu}(x) &= (\beta_0^J)^{-1} (1-x)^{\alpha} (1+x)^{\beta} \dx{x} + \sum_{j=1}^M \nu_j \delta_{\tau_j} \dx{x}, & \alpha, \beta &> -1, & \nu_j &> 0,
\end{align}
where $\beta_0^J = \int_{-1}^1 (1-x)^{\alpha} (1+x)^{\beta} \dx{x}$. The orthogonal polynomials belonging to the measure \eqref{eq:multiple} are explicitly known only in very special cases. The case of one mass point at one end point, that is, $M = 1, \tau_1 = -1$, has been studied and the recurrence coefficients can be computed with rather technical formulas \cite{chihara2011introduction,gautschi1994algorithm}. The exact recursion coefficients for $N = 1, 7, 18, 40$ are given in \cite[Table 2.11]{gautschi_orthogonal_2004}. For each of these particular $N$, we compute the fixed-$N$ error, donated by $e_N^{f} = \left(\left(a_N - \hat{a}_N\right)^2 + \left(b_N - \hat{b}_N\right)^2\right)^{1/2}$.

Table \ref{tab:multi} shows results for the \HD{}, \aPC{}, \MC{}, \SP{}, and \PC{} approaches for the measure $\mu$ above. In addition, we compute results using the \LZ{} approach; note that the \LZ{} approach cannot directly be utilized on the measure \eqref{eq:multiple} since this measure has an infinite number of support points. Instead, the \LZ{} results shown in Table \ref{tab:multi} first use the discretization approach as described in Section \ref{subsec:LZ}, which replaces the continuous part of $\mu$ with a discrete Gaussian quadrature measure. The reason we include this test in Table \ref{tab:multi} is that it motivates the \PCL{} algorithm: if one can discretize measures, then the \LZ{} approach is frequently more accurate than alternative methods.


\begin{table}
  \begin{center}
  \resizebox{0.95\textwidth}{!}{
    \renewcommand{\tabcolsep}{0.4cm}
    \renewcommand{\arraystretch}{1.3}
    {\scriptsize
      \begin{tabular}{c|cc|cc|cc|cc}
      \toprule
        Method & \multicolumn{2}{c}{$N = 1$} & \multicolumn{2}{c}{$N =7$} & \multicolumn{2}{c}{$N = 18$} & \multicolumn{2}{c}{$N = 40$} \\\midrule
         \HD & 3.71e-14 & 2.22e-11 & 3.64e-12 & 1.81e-09 & 1.72e-04 & --- & --- & --- \\
         \aPC & 3.71e-14 & 2.22e-11 & 3.54e-12 & 1.81e-09 & 1.67e-04 & --- & --- & --- \\
         \MC & 3.71e-14 & 2.22e-11 & 3.63e-12 & 8.90e-11 & 3.02e-12 & 2.90e+00 & 3.87e-12 & 1.84e+00 \\
         \SP & 3.71e-14 & 2.22e-11 & 3.63e-12 & 5.44e-13 & 3.03e-12 & 3.80e-12 & 3.90e-12 & 2.48e-06 \\
         \LZ & 3.70e-14 & 2.22e-11 & 3.63e-12 & 5.44e-13 & 3.03e-12 & 3.80e-12 & 3.90e-12 & 2.10e-12 \\
         \PC & 3.71e-14 & 2.22e-11 & 3.63e-12 & 5.44e-13 & 3.02e-12 & 3.80e-12 & 3.90e-12 & 2.49e-06 \\
      \bottomrule
      \end{tabular}
    }
 }
  \end{center}
  \caption{Example for Section \ref{subsec:multi}: errors $e_{N}^{f}$ with one mass at $\tau_1 = -1$ with $\nu_1 = 0.5$ (subcolumns on the left) and $\tau_1 = 2$ with $\nu_1 = 1$ (subcolumns on the right). Here --- means a NaN value due to the numerical overflow from the instability of the corresponding method.}\label{tab:multi}
\end{table}

We generate the first 40 recursion coefficients for $\alpha = -0.6, \beta = 0.4$ of the Jacobi parameters in two cases: one mass at $\tau_1 = -1$ with strength $\nu_1 = 0.5$ and a single mass point of strength $\nu_1 = 1$ at $\tau_1 = 2$. The results, produced by routine \texttt{ex\_multi\_component.py} from \cite{Univariate_ttr_examples}, are shown in Table \ref{tab:multi}. \SP, \LZ{}, \PC{} and even \MC{} produce essentially identical results within machine precision in the first case. However, matters change significantly when a mass point is placed outside $[-1, 1]$, regardless of whether or not the other mass points on $[-1, 1]$ are retained \cite[Example 2.39]{gautschi_orthogonal_2004}. \SP{} and \PC{} become extremely unstable; this empirical superiority of the \LZ{} approach for discrete measures is the reason why the last step of the \PCL{} algorithm in Section \ref{subsec:PCL} is to utilize the Lanczos algorithm.

\subsection{General multiple component: continuous weight function plus a discrete measure}\label{subsec:gmulti}

In the previous example, we studied the case of a combination of Chebyshev weight and discrete measure. A quadrature for Chebyshev is trivial because it is one of the classical weights so that we can obtain the quadrature by known recursion coefficients. However, if the continuous weight is not of classical form, then we employ the \PCL{} algorithm in Section \ref{subsec:PCL}: We use \PC{} to compute recursion coefficients, leading to Gaussian quadrature nodes and weights for the continuous part, which is then combined with the discrete part as input to the \LZ{} algorithm. 

We consider the positive half-range Hermite measure plus a transformed discrete Chebyshev measure defined on $(-1,0]$,
\begin{align*}
\dx{\mu}(x) &= e^{-x^2} + \sum_{j=1}^{M} \nu_j \delta_{\tau_j} \dx{x}, & \tau_j &\coloneqq -\frac{j-1}{M}, & \nu_j &\coloneqq \frac{1}{M}.
\end{align*}

\begin{table}
  \begin{center}
  \resizebox{0.95\textwidth}{!}{
    \renewcommand{\tabcolsep}{0.4cm}
    \renewcommand{\arraystretch}{1.3}
    {\scriptsize
      \begin{tabular}{c|cc|cc|cc|cc|cc}
      \toprule
        $M$ & \multicolumn{2}{c}{$N = 20$} & \multicolumn{2}{c}{$N = 40$} & \multicolumn{2}{c}{$N = 60$} & \multicolumn{2}{c}{$N = 80$} & \multicolumn{2}{c}{$N = 100$} \\\midrule
         $20$ & 1.09e-14 & 7.47e-15 & 6.48e-14 & 1.63e-14 & 1.46e-10 & 6.61e-13 & 2.41e-03 & 5.63e-12 & 1.66e+07 & 3.27e-09 \\
         $40$ & 6.50e-15 & 1.05e-14 & 2.50e-14 & 3.28e-14 & 9.34e-11 & 9.52e-14 & 8.54e-03 & 1.84e-13 & 1.95e+09 & 3.05e-11 \\
         $80$ & 8.80e-15 & 5.11e-15 & 1.39e-14 & 4.74e-14 & 1.68e-11 & 3.90e-14 & 4.48e-03 & 8.97e-14 & 5.10e+08 & 4.95e-11 \\
         $160$ & 7.73e-15 & 7.13e-15 & 1.43e-14 & 3.99e-14 & 2.90e-11 & 7.03e-14 & 1.88e-03 & 1.24e-13 & 2.34e+09 & 2.25e-11 \\
         
      \bottomrule
      \end{tabular}
    }
 }
  \end{center}
  \caption{Example for Section \ref{subsec:gmulti}: errors $f_N$ by procedure in \ref{subsec:PCL} with $N_s = N$ for all $s$ (subcolumns on the left) and by \PCL{}, i.e. with a adaptive procedure (subcolumns on the right) when $M = 20, 40, 80, 160$.}\label{tab:gmulti}
\end{table}

Using the \PCL{} algorithm, for $M = 20, 40, 80, 160$, we generate the first 100 recursion coefficients. Table \ref{tab:gmulti} shows that the coefficients are more accurate when an adaptive procedure is applied to determine $N_s$, no matter what $M$ is. The results here are produced by routine \texttt{ex\_gmulti\_component.py} in \cite{Univariate_ttr_examples}.

%

\section{Summary and extensions}\label{summary}

In this paper, we summarize several existing numerical methods for computing these recurrence coefficients associated to measures for which explicit formulas are not available. We propose a novel ``predictor-corrector" algorithm and study the accuracy and efficiency by comparing with existing methods for fairly general measures. The method makes predictions for the next coefficients and correct them iteratively. Finally, we introduce a hybrid algorithm that combines the ``predictor-corrector" algorithm and the (stabilized) Lanczos procedure. It can be used to compute recurrence coefficients for a general measure with multiple continuous and discrete components.

The predictor-corrector algorithm outperforms many other methods and is competitive with the Stieltjes procedure when a continuous measure is given. For a discrete measure, it can compute accurate coefficients only when the discrete support $M$ is large enough. However, the (stabilized) Lanczos procedure requires empirically appears to be superior for discrete measures. Based on this observation, we propose a ``predictor-corrector-Lanczos" algorithm is that is a hybrid of the predictor-corrector and Lanczos schemes, and applies to a fairly general class of measures.

We focus on the computation of recurrence coefficients for univariate orthogonal polynomial families. Thus, a natural extension of this work would be to adapt the approaches to address the same problem for multivariate polynomials, for which the formulations can be substantially more complex. Such investigations are the focus of ongoing work. 

\bibliographystyle{amsplain}
\bibliography{references.bib}

\end{document}